\documentclass[letterpaper]{article}

\usepackage{amssymb, amsfonts, amsmath,amsthm}
\usepackage[dvips]{graphicx}

\DeclareMathOperator{\mult}{mult}

\DeclareMathOperator{\rk}{rk}
\DeclareMathOperator{\edim}{edim}

\DeclareMathOperator{\red}{red}

\DeclareMathOperator{\lex}{lex}
\DeclareMathOperator{\rlex}{rlex}

\newcommand{\Q}{\mathbb{Q}}

\newcommand{\N}{\mathbb{N}}
\newcommand{\proj}{\mathbb{P}}
\newcommand{\ocal}{\mathcal{O}}

\theoremstyle{plain}
\newtheorem{theorem}{Theorem}
\newtheorem{lemma}[theorem]{Lemma}

\newtheorem{proposition}[theorem]{Proposition}
\newtheorem{corollary}[theorem]{Corollary}
\newtheorem{algorithm}[theorem]{Algorithm}

\theoremstyle{definition}
\newtheorem{definition}[theorem]{Definition}

\theoremstyle{remark}
\newtheorem{example}[theorem]{Example}

\newtheorem{remark}[theorem]{Remark}

\numberwithin{theorem}{section}
\numberwithin{equation}{section}
\numberwithin{figure}{section}

\begin{document}
\title{New methods for determining speciality of linear systems based at fat points in $\proj^n$}
\author{Stepan Paul}
\maketitle

\begin{abstract}
In this paper we develop techniques for determining the dimension of linear systems of divisors based at a collection of general fat points in $\proj^n$ by partitioning the monomial basis for $H^0(\ocal_{\proj^n}(d))$. The methods we develop can be viewed as extensions of those developed by Dumnicki. We apply these techniques to produce new lower bounds on multi-point Seshadri constants of $\proj^2$ and to provide a new proof of a known result confirming the perfect-power cases of Iarrobino's analogue to Nagata's Conjecture in higher dimension.
\end{abstract}

\vskip18pt

Let $\mathbb K$ be a field of characteristic $0$. Given general points $p_1,\ldots, p_r$ in $\proj^n=\mathbb K\proj^n$, some $\mathbf m=(m_1,\ldots,m_r)\in\N^r$, and $d\geq1$, we are interested in determining whether there exists a degree-$d$ hypersurface with multiplicity at least $m_i$ at $p_i$ for all $i$. It is an open problem to formulate a general, definitive, and computationally succinct method for answering this question. 

Let $V^n(d)$ be the vector space of homogeneous degree-$d$ polynomials in $n+1$ variables, and, let $V^n(d,\mathbf m)$ be the subspace consisting of polynomials which vanish with the prescribed multiplicities at general points. Then we may approach the problem by seeking conditions which determine when the data of $(n,d,\mathbf m)$ is \emph{special}; that is, when it fails to satisfy
$$\dim V^n(d,\mathbf m)={d+n\choose n}-\sum{m_i+n-1\choose n}.$$
Determining these conditions is an area of active research, with many partial results, conjectures, and computational techniques.

For $n=2$, Nagata studied the homogeneous case of $\mathbf m=(m,\ldots,m)$ in \cite{MR0088034} in order to produce a counterexample to Hilbert's 14th Problem. In that paper, he formulated the conjecture bearing his name, which states that if $r>9$, $d^2\leq rm^2$, then $V^2(d,(m,\ldots,m))=0$, and proved the case where $r$ is a perfect square. The Harbourne-Hirschowitz Conjecture \cite{MR846019,MR993223} generalizes Nagata's by proposing specific criteria for $(2,d,\mathbf m)$ to be special (see \cite{MR2098342} for a nice explanation). Papers by Dumnicki and Jarnicki prove the Harbourne-Hirschowitz Conjecture for homogeneous multiplicities $m\leq42$ in \cite{MR2289179}, and the general case with all multiplicities no more than $11$ in \cite{MR2325918}. The $2$-dimensional case is of particular interest because of its connection to multi-point Seshadri constants of $\proj^2$ and the problem of determining the ample  cone for rational surfaces (see for example \cite{MR2555949}).

For $n\geq3$, there are some related conjectures, including one from Iarrobino proposing a higher dimensional analogue of Nagata's Conjecture in \cite{MR1337187}, which we address in Section \ref{results}. In any dimension, the problem of determining speciality is also applicable to Hermite interpolation (see for example \cite{MR2342565}).

Given any particular set of data $(n,d,\mathbf m)$, there is a definitive ``brute force'' matrix rank computation for determining $\dim V^n(d,\mathbf m)$, which we summarize in Section \ref{linalgsec}.
In \cite{MR2289179,MR2342565}, Dumnicki introduces the idea of taking nested subsets of the monomial basis of $V^n(d)$, which we index by a set $D(d)$, to recursively reduce this computation to finding the ranks of smaller matrices, many of which are shown to be nonsingular \emph{a fortiori} by combinatorial arguments.

In Sections \ref{partitionssec} and \ref{generalizedreductionalgorithms}, we prove Theorems \ref{maintheorem} and \ref{genredalgtheorem}, which show that these procedures can be generalized by instead considering a larger class of partitions of $D(d)$.

In Section \ref{constructions} we describe some constructions for these partitions, expand on Dumnicki's combinatorial criteria for nonsingularity, and thereby produce new algorithms for determining speciality. Our two main theoretical tools here for constructing partitions of the correct kind are Theorem \ref{corfirst} and Lemma \ref{anynaf}. The former constructs partitions using monomial orderings on $V^n(d)$, and the latter describes a useful class of partitions which satisfy the combinatorial conditions.

In Section \ref{results} we apply these methods towards producing new bounds on multi-point Seshadri constants of $\proj^2$ (summarized in Figure \ref{lowerbounds}), and recovering a theorem of Evain from \cite{MR2125451} which proves the perfect-power cases of Iarrobino's Conjecture, generalizing the known perfect-square cases of Nagata's Conjecture to higher dimensions.

\section{The Linear Algebra Set-Up}\label{linalgsec}

We will work over a field $\mathbb K$ of characteristic $0$. Given $m\geq0$, $n\geq 1$, $d\geq1$, there exists a natural sheaf homomorphism
$$\Pi(n,d,m):\mathcal J^m(d)\longrightarrow\ocal_{\proj^n}(d)$$
where $\mathcal{J}^m(d)$ is the $m$th jet bundle of $\ocal_{\proj^n}(d)$.

Define $V^n(d)=H^0\ocal_{\proj^n}(d)$. Then by $\gamma(n,d,m)$ we denote the natural prolongation map
$$\gamma(n,d,m):V^n(d)\longrightarrow H^0\mathcal J^m(d).$$

To aid notation, if $\mathbf a=(a_0,\ldots,a_n)\in\N^{n+1}$, we define $|\mathbf a|=a_0+\cdots+a_n$, $D(d)=\left\{\left.\mathbf a\in\N^{n+1}\right||\mathbf a|=d\right\}$, and $\mathbf{X^a}=X_0^{a_0}\cdots X_n^{a_n}$. We also notice that $D(d)$ can be visualized as the integral points of $d$ times the standard $n$-simplex. For $n=2$, we can illustrate $D(d)$ in a triangle as shown in Figure \ref{d7}. If $n$ is unclear from context, we may write $D^n(d)$.

\begin{figure}
\centering
\includegraphics{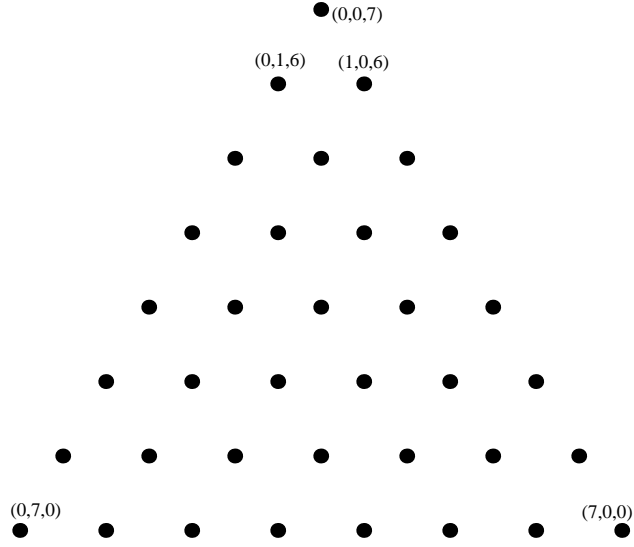}
\caption{An illustration of $D(7)$. In subsequent illustrations, we will omit the vertex labels.}\label{d7}
\end{figure}

We can identify $V^n(d)$ with the vector space of homogeneous degree-$d$ polynomials in $\mathbb K[\mathbf X]=\mathbb{K}[X_0,\ldots,X_n]$, a $d+n\choose n$-dimensional space. Note that $V^n(d)$ has a natural basis consisting of monomials $\left\{\mathbf{X^a}\left|\mathbf a\in D(d)\right.\right\}$.

Also, we can identify $H^0\mathcal J^{m}(d)$ with a subspace of
$$V^n(d-m+1)\otimes_{\mathbb K}\mathbb K^{D(m-1)}$$
by thinking of a section of the jet bundle $\mathcal J^{m}(d)$ as the tuple of all order-$(m{-}1)$ partial derivatives, which are indexed by $D(m-1)$, of a polynomial in $V^n(d)$. Each of these derivatives is a homogeneous degree-$(d{-}m{+}1)$ polynomial, say in $\mathbb K[\mathbf P]$.

Given these identifications, $\gamma(n,d,m)$ sends a homogeneous degree-$d$ polynomial to the ordered set of its order-$(m{-}1)$ partial derivatives.
Specifically,
$$\gamma(n,d,m)(f)=\left(\left.\frac{\partial^{|\mathbf b|}f}{\mathbf{\partial X^b}}(\mathbf P)\right|\mathbf b\in D(m-1)\right).$$

Now $\gamma(n,d,m)$ is represented by the matrix $M(m)$ with columns indexed by $D(d)$, rows indexed by $D(m-1)$, and polynomial entries in $V^n(d-m+1)$ via
\begin{equation}\label{defineM}
M(m)_{[\mathbf a,\mathbf b]}=\frac{\partial^{|\mathbf b|}\mathbf{X^a}}{\mathbf{\partial X^b}}(\mathbf P)=\left(\prod_{j=0}^{b_1-1}a_1-j\right)\cdots\left(\prod_{j=0}^{b_n-1}a_n-j\right)\mathbf P^{\mathbf{a-b}}.
\end{equation}

For any $p\in\proj^n$, there is a natural evaluation map
$$\nu_{p}:H^0\mathcal J^{m}(d)\longrightarrow J^m_{p}(d)$$
where $J^m_{p}(d)$ is the jet space over $p$---that is, the fiber of $\mathcal J^m(d)$ at $p$.

Now, if $r\geq1$ and $\mathbf m=(m_1,\ldots,m_r)\in\N^r$, let $\mathcal{J}=\bigoplus_{i=1}^r\mathcal{J}^{m_i}(d)$, and define $\Pi$ to be the sheaf homomorphism
$$\Pi=\sum_{i=1}^r\Pi(n,d,m_i):\mathcal{J}\longrightarrow\ocal_{\proj^n}(d).$$
Therefore we have $H^0\mathcal J=\bigoplus_{i=1}^r H^0\mathcal J^{m_i}(d)$. To keep track of the fact that the sum is direct, we put an additional index on the indeterminates of polynomials in $H^0\mathcal J^{m_i}(d)$. In particular, we think of $H^0\mathcal J^{m_i}(d)$ as consisting of tuples of polynomials in $\mathbb K[P_{i,0},\ldots,P_{i,n}]=\mathbb K[\mathbf P_i]$ so that $M(m_i)$ is a matrix with entries in $\mathbb K[\mathbf P_i]$.

We then define
$$\gamma=\left(\gamma(n,d,m_1),\ldots,\gamma(n,d,m_r)\right):V^n(d)\longrightarrow H^0\mathcal J.$$
Then $\gamma$ can then be represented by the matrix $M=M(\mathbf m)$ defined by
$$M=\left[\begin{array}{c}
M(m_1)\\
\vdots\\
M(m_r)
\end{array}\right].$$

To aid notation, we let $U_i=\{(i,\mathbf b)|\mathbf b\in D(m_i-1)\}$ and $U=\bigcup_{i=1}^r U_i$ so that the rows of $M$ are indexed by $U$. In particular, we can say 
$$M_{[\mathbf a,(i,\mathbf b)]}=M(m_i)_{[\mathbf a,\mathbf b]}.$$

Basically, $M$ multiplies the coefficient vector of a polynomial in $V^n(d)$ on the right, and yields a collection of polynomials, which are indexed by $U$, in the variables $\{P_{i,j}\}$.

For any $r$-tuple of points $p_1,\ldots,p_r\in\proj^n$, we have an evaluation map
$$\nu_{p_1,\ldots,p_r}:H^0\mathcal J\longrightarrow\bigoplus_{i=1}^r J^{m_i}_{p_i}(d)$$
whose components are the evaluation maps $\nu_{p_i}$ on $H^0\mathcal J^{m_i}(d)$.

Define $V^n(d,(m_1p_1,\ldots,m_rp_r))$ to be the kernel of $\nu_{p_1,\ldots, p_r}\gamma$---the space of sections of $\ocal_{\proj^n}(d)$ vanishing with multiplicity at least $m_i$ at $p_i$ for each $i$. If the points $p_i$ are taken to be general, we suppress them in the notation as $V^n(d,(m_1,\ldots,m_r))=V^n(d,\mathbf m)$. 

\begin{definition}
We say a section $f=\sum_{\mathbf a\in D(d)}\kappa_{\mathbf a}\mathbf{X^a}\in V^n(d)$ is \emph{supported} on a subset $D$ of $D(d)$ if $\kappa_{\mathbf a}=0$ whenever $\mathbf a\notin D$. We let $V^n_D$ be the subspace of $V^n(d)$ of those sections supported on $D$. We make the analogous definitions for $V^n_D(m_1p_1,\ldots,m_rp_r)$ and $V^n_D(m_1,\ldots,m_r)$.
\end{definition}

We also define $\gamma_D=\gamma|_{V^n_D(\mathbf m)}$, which is represented by the sub-matrix $M_D(\mathbf m)$ of $M(\mathbf m)$ containing only those columns indexed by elements $\mathbf a\in D$. We then have 
$$\ker (\nu_{p_1,\ldots,p_r}\gamma_D)=V_D^n(d,(m_1p_1,\ldots,m_rp_r)).$$

The following proposition allows us to study $V^n_D(\mathbf m)$, where the points are taken to be general, by focusing on the matrix $M_D(\mathbf m)$ with polynomial entries.

\begin{proposition}
For general points $p_1,\ldots,p_r\in\proj^n$, 
$$\rk\gamma_D=\rk (\nu_{{p_1},\ldots,{p_r}}\gamma_D).$$
\end{proposition}

If we de-homogenize the system---say, by setting the $X_i$ coordinate to $1$---then the $n=2$ case is Dumnicki's Proposition 9 in \cite{MR2289179}. The proof here is essentially the same.

\begin{proof}
Since $\nu_{p_1,\ldots,p_r}$ is a homomorphism, it suffices to prove that $\rk\gamma_D\leq\rk(\nu_{p_1,\ldots,p_r}\gamma_D)$. 

The rank of $\gamma_D$ is the size of the largest minor of $M=M_D(\mathbf m)$ which is not identically zero as a polynomial---call this polynomial $\mu$. The evaluation of $\mu$ at general nonzero points $\hat{p_1},\ldots,\hat{p_r}\in\mathbb K^{n+1}$ is then also nonzero. 

Let $\hat M$ be the matrix with scalar entries obtained by evaluating each entry of $M$ at the points $\hat{p_1},\ldots,\hat{p_r}$. Letting $p_i$ be the point in $\proj^n$ over which $\hat{p_i}$ lies, we can (non-canonically) identify 
$\bigoplus_{i=1}^r J^{m_i}_{p_i}(d)$ with $\mathbb K^U$ so that $\hat M$ is the matrix representing $\nu_{{p_1},\ldots,{p_r}}\gamma_D$.

We then see that the corresponding minor of $\hat M$ is exactly $\mu(\hat{p_1},\ldots,\hat{p_r})$, which is known to be nonzero, and so $\hat M$ has at least the same rank as $M$.
\end{proof}

\begin{corollary}\label{rankM}
In general, we have
$$\dim V_D^n(\mathbf m)=\# D-\rk M_D(\mathbf m).$$
\end{corollary}

We will use the word \emph{triple} to refer to the data of $(n,D,\mathbf m)$ with the understanding that $n\geq1$, $D\subseteq D^n(d)$ for some $d\geq 1$, and $\mathbf m\in\N^{r}$ for some $r\geq 1$.

\begin{definition}
We call a triple $(n,D,\mathbf m)$ \emph{non-special} if the following equivalent conditions are met.
\begin{enumerate}
\item $M_D(\mathbf m)$ has full rank;
\item If $\dim V^n_D(\mathbf m)$ has the \emph{expected dimension} of 
$$\edim(n,D,\mathbf m)=\max\left\{\#D-\#U,0\right\}.$$
\end{enumerate}
A triple is \emph{special} if it is not non-special. If $n$ and $\mathbf m$ are understood, we may call $D$ special or non-special as well.
\end{definition}

Notice that this definition specializes to the one given in the introduction when $D=D(d)$ since $\#D(d)={d+n\choose n}$.

\begin{remark}
We point out that the definition splits depending on the sign of $\#D-\#U$. In particular
\begin{enumerate}
\item if $\#D<\#U$, then we say $(n,D,\mathbf m)$ is \emph{over-determined}, and it is non-special if and only if $V^n_D(\mathbf m)=0$;
\item if $\#D>\#U$, then we say $(n,D,\mathbf m)$ is \emph{under-determined}, and it is non-special if and only if $\dim V^n_D(\mathbf m)=\#D-\#U$;
\item if $\#D=\#U$, then we say $(n,D,\mathbf m)$ is \emph{well-determined}, and is non-special if and only if $\dim V^n_D(\mathbf m)=\#D-\#U=0$.
\end{enumerate}
By definition, we always have 
\begin{equation}\label{intersection}
V^n(d,(m_1p_1,\ldots,m_rp_r))=\bigcap_{i=1}^rV^n(d,m_ip_i).
\end{equation}
Another characterization of speciality for under- or well-defined triples is that a triple is non-special exactly when there are points $p_i$ general enough so that the codimension of the lefthand side is equal to the sum of the codimensions of the spaces being intersected on the righthand side.
\end{remark}

\section{Partitions of Monomials}\label{partitionssec}

Here we present a generalization of Dumnicki and Jarnicki's notion of ``reduction'' from \cite{MR2289179, MR2543429, MR2325918}. The content of this generalization is that instead of reducing one point at a time, we may reduce by several at once. Our notation will also differ slightly from the papers cited because we do not de-homogenize our polynomials by choosing an affine chart. Instead we opt to preserve the symmetry afforded by working over all of $\proj^n$, which will be put to use in Section \ref{constructions}.

As a bit of notation, if $A$ is any matrix with rows indexed by $I$ and columns indexed by $J$, we will write $(I',J')$ to denote the sub-matrix with rows in $I'\subseteq I$ and columns in $J'\subseteq J$. As a convention, we will set $\det(\varnothing,\varnothing)=1$.

\begin{lemma}[A Generalized Laplace Rule (GLR)]\label{glr}
Let $A$ be any square matrix with rows indexed by (an ordered set) $I$ and columns indexed by (an ordered set) $J$, and let $(I_1,\ldots,I_s)$ be a partition of $I$. Let $\mathcal{P}$ be the set of partitions $(J_1,\ldots,J_s)$ of $J$ with $\#I_i=\#J_i$ for all $i$. Then
\begin{equation}
\det A=\sum_{(J_1,\ldots,J_s)\in\mathcal{P}}\pm\left(\prod_{i=1}^s\det(I_i,J_i)\right)\label{glreqn}
\end{equation}
\end{lemma}

\begin{proof}
Recursively use the Generalized Laplace Rule for $s=2$.
\end{proof}

Given a triple $(n,D,\mathbf m)$, define $U$, $U_i$, and $M=M_D(\mathbf m)$ as above. Then let $(U',D')$ be some square sub-matrix of $M$, and define $U_i':=U'\cap U_i$. Finally let $\mathcal P(U',D')$ be the set of partitions $\mathbf E=(E_1,\ldots,E_r)$ of $D'$ with $\#E_i=\#U_i'$ for all $i$. Then the GLR gives us
\begin{equation}\label{applyglr}
\det(U',D')=\sum_{\mathbf E\in\mathcal P(U',D')}\pm\left(\prod_{i=1}^r\det(U_i',E_i)\right).
\end{equation}

In this situation, we will refer to the summand associated to $\mathbf E$ in (\ref{applyglr}) as $\sigma(\mathbf E)$. For any $\mathbf E\in\mathcal P(U',D')$, we can compute directly from (\ref{defineM}) that, for some scalar $\kappa$,
\begin{eqnarray}\label{summandformula}
\det(U_i',E_i)=\kappa\mathbf P_i^{\mathbf a_i(\mathbf E)-\mathbf b_i}\\
\nonumber\\
\mathbf a_i(\mathbf E):=\sum_{\mathbf a\in E_i}\mathbf a,\qquad \mathbf b_i:=\sum_{(i,\mathbf b)\in U_i'}\mathbf b\nonumber
\end{eqnarray}
In particular, (\ref{summandformula}) is either zero or has one term as a polynomial. Hence $\sigma(\mathbf E)$ is some scalar multiple of the monomial
\begin{equation}\label{theterm}
\mathbf P_1^{\mathbf a_1(\mathbf E)-\mathbf b_1}\cdots \mathbf P_r^{\mathbf a_r(\mathbf E)-\mathbf b_r}
\end{equation}
Notice that the $\mathbf b_i$ depend only on the choice of $U'\subseteq U$, and not on the partition $\mathbf E$.

\begin{definition}\label{exceptional}
Given a triple $(n,D,\mathbf m)$ and a square sub-matrix $(U',D')$ of $M$, we call a partition $\mathbf E\in\mathcal P(U',D')$ \emph{exceptional} (with respect to $(U',D')$) if it satisfies the properties
\begin{enumerate}
\item\label{nondegen} $\sigma(\mathbf E)\neq0$.
\item\label{unique} If $D''\subseteq D$ has $\#D''=\#D'$, and $\mathbf F\in\mathcal P(U',D'')$ is a different partition with $\mathbf a_i(\mathbf F)=\mathbf a_i(\mathbf E)$ for all $i$, then $\sigma(\mathbf F)=0$.
\end{enumerate}

If additionally, $(U',D')$ is a maximal square sub-matrix of $M$, then we call $\mathbf E$ a \emph{fully exceptional partition}.

In the case where $U'\subseteq U_r$, so that $\mathbf E=(\varnothing,\ldots,\varnothing,E_r)$, we call the partition (or just $E_r$) a \emph{reduction}.
\end{definition}

\begin{remark}
Notice that if $(n,D,\mathbf m)$ is over- (respectively under-, well-) determined, then $(U',D')$ is maximal if and only if $D'=D$ (respectively $U'=U$, $(U',D')=M$).

One fact to keep in mind is that if $\#E_i=\#F_i$, then $\mathbf a_i(\mathbf E)=\mathbf a_i(\mathbf F)$ if and only if the centroid of the points in $E_i$ is the same as the centroid of the points in $F_i$. This is a visual trick which may be helpful for looking at examples.
\end{remark}

We now state the first main theoretical result.

\begin{theorem}\label{maintheorem}
Suppose $(n,D,\mathbf m)$ admits an exceptional partition with respect to $(U',D')$ with $U'\subseteq U_{k+1}\cup\cdots\cup U_r$ for some $k\leq r$. Then
$$\dim V^n_D(m_1,\ldots,m_r)\leq\dim V^n_{D\smallsetminus D'}(m_1,\ldots,m_k).$$
In particular, if $k=0$, then $\dim V^n_D(\mathbf m)\leq\#(D\smallsetminus D')$.
\end{theorem}

We list some special cases in the following corollary.

\begin{corollary}\label{consequences}~
\begin{enumerate}
\item A triple which admits a fully exceptional partition is non-special.
\item If $(n,D(d),\mathbf m)$ is over- or well-determined and admits a fully exceptional partition, then for general points $p_1,\ldots, p_r$, the linear series in $\ocal_{\proj^n}(d)$ of hyper-surfaces with multiplicity $m_i$ at $p_i$ for all $i$ is empty.
\end{enumerate}
\end{corollary}

We will use the abbreviated notation 
$$m^{\times r}=(\;\underbrace{m,\ldots,m}_{\text{$k$ times}}\;)$$

\begin{example}\label{exceptionalnotreduction}
Here we apply Corollary \ref{consequences}.2 to show that no degree $7$ curve in $\proj^2$ has multiplicity $3$ at each of 6 general points. That is, we show $V^2(7,3^{\times 6})=0$. We claim that Figure \ref{notredpic} illustrates a fully exceptional partition, call it $\mathbf E$, of $D(7)$.
\begin{figure}
\centering
\includegraphics{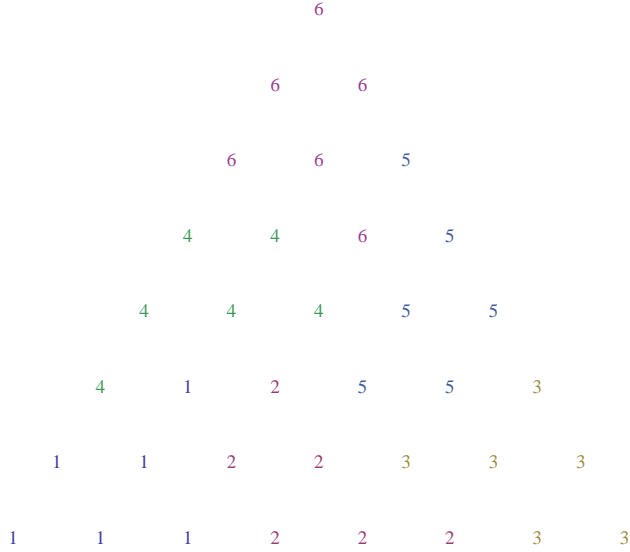}
\caption{An exceptional partition for $(2,D(7),3^{\times 6})$. $E_i$ consists of the points marked with $i$, for $i=1,2,3,4,5,6$.}\label{notredpic}
\end{figure}
That $\sigma(\mathbf E)\neq0$ follows from Corollary \ref{af} below. That no other partition $\mathbf F$ has $\mathbf a_i(\mathbf E)=\mathbf a_i(\mathbf F)$ for $1\leq i\leq6$ can be checked exhaustively, or eyeballed by observing that no other partition $\mathbf F$ has the same sextuple of centroids of its parts as $\mathbf E$. 

(Begin by noticing that there are only $3$ possible sets of $6$ points with the same centroid as $E_1$. For each of these, there are $6$ or $7$ possible sets of $6$ points with disjoint from the first with the same centroid as $E_2$. Among the $20$ cases you end up with, only $3$ allow for a set of $6$ points disjoint from the first two sets with the same centroid as $E_3$. Finally, among these $3$ possibilities, only one admits a set of $6$ point disjoint from the other sets with the same centroid as $E_4$, and that is the case that is shown. By symmetry, we have shown uniqueness.)
\end{example}

\begin{proof}[Proof of Theorem \ref{maintheorem}]
We start by noting that it suffices to prove
\begin{equation}\label{rankstatement}
\rk M_D(m_1,\ldots,m_r)\geq\rk M_{D\smallsetminus D'}(m_{k+1},\ldots,m_r)+\#D'
\end{equation}

Let $(U'',D'')$ be a maximal nonsingular submatrix of 
$$M_{D\smallsetminus D'}(m_{k+1},\ldots,m_r)=(U_{k+1}\cup\cdots\cup U_r,D\smallsetminus D').$$
Then let $\mathcal P$ be the set of partitions $(C',C'')$ of $D'\cup D''$ with $\#C'=\#U'$ and $\#C''=\#U''$. Then by the GLR we have
$$\det (U'\cup U'',D'\cup D'')=\sum_{(C',C'')\in\mathcal P}\pm\det(U',C')\det(U'',C'').$$
Again applying the GLR, we have
$$\det(U',C')=\sum_{\mathbf F\in\mathcal P(U',C')}\pm\left(\prod_{i=1}^r\det(U_i',F_i)\right).$$
Combining these, we get
\begin{equation}\label{expansion}
\det(U'\cup U'',D'\cup D'')=\sum_{\substack{(C',C'')\in\mathcal P,\\
\mathbf F\in\mathcal P(U',C')}}\pm\left(\det(U'',C'')\prod_{i=1}^k\det(U_i',F_i)\right).
\end{equation}

We claim that the only summand of (\ref{expansion}) containing nonzero terms divisible by $\sigma(\mathbf E)$ is the one corresponding to $(D',D'')\in\mathcal P,\;\mathbf E\in\mathcal P(U',D')$. Furthermore we note that this summand is nonzero by the assumption that $(D'',U'')$ is nonsingular and that $\sigma(\mathbf E)\neq0$. If the claim is true, these terms cannot cancel with terms from other summands, and so the determinant in (\ref{expansion}) is nonzero as a polynomial. That is, $\rk M_D(\mathbf m)\geq\#(U'\cup U'')$, proving (\ref{rankstatement}).

To prove the claim, first notice that $\det(U'',C'')$ is a polynomial in $\{P_{i,j}|i>k\}$ and $\det(U',C')$ is a polynomial in $\{P_{i,j}|i\leq k\}$. Hence a nonzero summand of (\ref{expansion}) contains terms divisible by $\sigma(\mathbf E)$ if and only if $\prod_{i=1}^k\det(U_i',F_i)$ does. And, by the exceptionality of $\mathbf E$, this product is a nonzero multiple of $\sigma(\mathbf E)$ if and only if $C'=D'$ and $\mathbf F=\mathbf E$.
\end{proof}

\section{Generalized Reduction Algorithms}\label{generalizedreductionalgorithms}

In order to obtain sharper results, we can make a slight generalization to Theorem \ref{maintheorem}.

\begin{corollary}\label{gaug}
Suppose $D\subseteq G\subseteq D(d)$, and $(n,G,\mathbf m)$ admits an exceptional partition with respect to some $(U',D')$ with $D'\supseteq G\smallsetminus D$ and $U'\subseteq U_{k+1}\cup\cdots\cup U_r$, $k\leq r$. Then
$$\dim V^n_D(m_1,\ldots,m_r)\leq\dim V^n_{D\smallsetminus D'}(m_1,\ldots,m_k).$$
In particular, if $k=0$, then $\dim V^n_D(\mathbf m)\leq\#(D\smallsetminus D')$.
\end{corollary}
\begin{proof}
Since $G\supseteq D$, we certainly have 
$$\dim V^n_D(m_1,\ldots,m_r)\leq\dim V^n_G(m_1,\ldots,m_r).$$
Then applying Theorem \ref{maintheorem}, and noting that $G\smallsetminus D'=D\smallsetminus D'$ by assumption, we get that
$$\dim V^n_G(m_1,\ldots,m_r)\leq\dim V^n_{D\smallsetminus D'}(m_1,\ldots,m_k).$$
\end{proof}

It is a slightly annoying point that we allow for the possibility that $G$ properly contains $D$. It is not even obvious that this allowance provides any additional information because we are essentially adding points to $D$ only to throw them away again. However, Example \ref{gnecessaryex} shows that the generalization is nontrivial.

\begin{example}\label{gnecessaryex}
In Figure \ref{gnecessaryfig}, we illustrate a reduction $D'$ of $G$ containing $G\smallsetminus D$. However, $D\cap D'$, is not a reduction of $D$. In particular, $D''$, which is also illustrated, has $\#D'=\#D''$, the same centroid as $D'$, and $(U'',D'')$ nonsingular for suitably chosen $U''$. These facts can be proved using Corollary \ref{nonspeccond} below.
\begin{figure}
\centering
\includegraphics{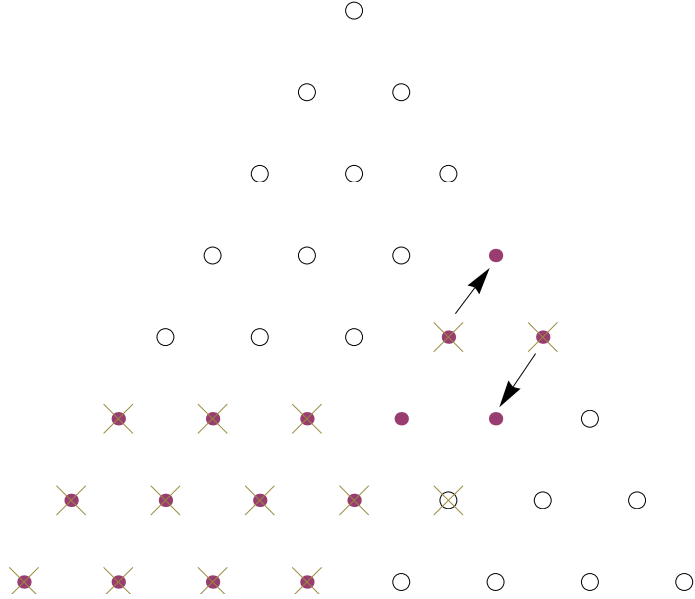}
\caption{A reduction $D'$ of $G$ whose intersection with $D$ is not a reduction of $D$. $D$ is labeled by solid points, $D'$ is labeled by $\times$'s, $G=D\cup D'$, and $D''$ is obtained by replacing the elements of $D'$ at the tails of the two arrows with the elements at the heads.}\label{gnecessaryfig}
\end{figure}
\end{example}

Pushing this generalization further, we may want to use Corollary \ref{gaug} recursively, which leads us to the following definition.

\begin{definition}\label{gradef}
A \emph{generalized reduction algorithm} for a triple $(n,D,\mathbf m)$ is a sequence of integers $0=r_0<r_1<\cdots<r_s=r$ and nested subsets $D=D_s\supseteq\cdots\supseteq D_1\supseteq D_0$ so that
for $1\leq i\leq s$, there exist $G_i\supseteq D_i$ so that $\left(n,G_i,\left(m_1,\ldots,m_{r_i}\right)\right)$ admits an exceptional partition with respect to $\left(U^{(i)},G_i\smallsetminus D_{i-1}\right)$ for some $U^{(i)}\subseteq U_{r_{i-1}+1}\cup\cdots\cup U_{r_i}$.

In the case where $s=r$ (so that each exceptional partition is a reduction), we simply call this a \emph{reduction algorithm}.

If $\#D_0=\edim(n,D,\mathbf m)$, then we call the (generalized) reduction algorithm \emph{full}.
\end{definition}

\begin{theorem}\label{genredalgtheorem}
If $(n,D,\mathbf m)$ admits a generalized reduction algorithm, then $\dim V^n_D(\mathbf m)\leq\#D_0$, where $D_0$ is as in Definition \ref{gradef}. In particular, a triple which admits a full generalized reduction algorithm is non-special.
\end{theorem}

\begin{remark}\label{comparedumredalg}
This is a generalization of Dumnicki and Jarnicki's notion of ``reduction algorithm'' from \cite{MR2325918}. Our primary innovation here is the case where $s<r$. That is, instead of reducing one point at a time, we may reduce by several points at once. Generalized reduction algorithms also generalize applications of Dumnicki's ``diagram-cutting'' method from \cite{MR2289179} for showing non-speciality.

The fully exceptional partition given in Example \ref{exceptionalnotreduction} above is a demonstration of why this is a nontrivial generalization; it cannot be produced one point at a time by a reduction algorithm. To see this, notice that no single part of the partition has a centroid which cannot arise at the centroid of another non-special collection of six points. That being said, the triple in question, $(2,D(7),3^{\times6})$, does admit a full reduction algorithm as illustrated by Figure \ref{redpic} (see Example \ref{alg1ex} below).

The author is not aware of a triple for which a full generalized reduction algorithm exists but a full reduction algorithm does not. In fact, it is apparently unknown if any non-special triples exist which do not admit full reduction algorithms (see Conjecture 19 in \cite{MR2325918}).
\end{remark}

\begin{proof}[Proof of Theorem \ref{genredalgtheorem}]
By Corollary \ref{gaug}, we know that for $1\leq i\leq s$,

$$\dim V^n_{D_i}\left(m_1,\ldots,m_{r_i}\right)\leq\dim V^n_{D_{i-1}}\left(m_1,\ldots,m_{r_{i-1}}\right).$$
Hence, we get that
$$\dim V^n_{D_{s}}(m_1,\ldots,m_r)\leq\dim V^n_{D_0}=\#D_0,$$
which proves the theorem.
\end{proof}

We note that a generalized reduction algorithm gives rise to a partition of $D\smallsetminus D_0$. As demonstrated by Example \ref{gnecessaryex}, the partition is not necessarily exceptional if $G_i$ properly contains $D_i$ for some $i$. However if $G_i=D_i$ for all $i$, the resulting partition is necessarily exceptional. 

This fact implies that the generalization afforded by Theorem \ref{genredalgtheorem} is only useful for proving non-speciality of over-determined triples. For an under- or well-determined triple, $\#D_0=\edim(n,D,\mathbf m)$ implies $G_i=D_i$ for all $i$.

Now that we have established Theorem \ref{genredalgtheorem}, we can focus on techniques for producing exceptional partitions and reductions. Section \ref{constructions} describes some criteria for $\sigma(\mathbf E)\neq0$ and constructions for reduction algorithms, and Section \ref{results} will use these constructions---as well as some \emph{ad hoc} methods---to build full generalized reduction algorithms for some interesting examples.

\section{Constructions}\label{constructions}

Let $\preceq$ be any monomial ordering on $\mathbb{K}[X_0,\ldots,X_n]$. Notice $\preceq$ induces an ordering on $\N^{n+1}$, which we will also call $\preceq$, via
$$\mathbf a\preceq \mathbf b\Longleftrightarrow\mathbf{X^a}\preceq\mathbf{X^b}.$$

For any $D\subseteq D(d)$, $d\geq 2$, and $c\geq1$, define 
$$\mathcal E(D,c)=\{E\subseteq D|\#E=c\}.$$

\begin{definition}
Suppose $E,F\in\mathcal E(D,c)$ with $E=\{\mathbf a_1,\ldots,\mathbf a_c\}$ and $F=\{\mathbf b_1,\ldots,\mathbf b_c\}$ with $\mathbf a_i\prec\mathbf a_{i+1}$ and $\mathbf b_i\prec \mathbf b_{i+1}$ for $1\leq i<c$. Then we define the \emph{$\preceq$-lexicographic ordering} on $\mathcal E(D,c)$, for which we will abuse notation and also call $\preceq$, by
$$E\prec F\Longleftrightarrow\text{ there exists $k\geq1$ so that $\mathbf a_k\prec\mathbf b_k$, and $\mathbf a_i=\mathbf b_i$ for all $i<k$.}$$
\end{definition}

\begin{lemma}
The $\preceq$-lexicographic ordering on $\mathcal E(D,c)$ is a well-ordering for any monomial ordering $\preceq$.
\end{lemma}
\begin{proof}
This is probably standard, and the proof works for the lexicographic ordering of finite subsets of any well-ordered set. In any event, if $\mathcal F\subseteq \mathcal E(D,c)$, then take the minimal element $\mathbf a_1$ appearing in any $E\in \mathcal F$, and let $\mathcal F_1\subseteq \mathcal F$ be the collection of all $E\in\mathcal F$ containing $\mathbf a_1$. Then recursively take the minimal element $\mathbf a_{i+1}$ different from $\mathbf a_1,\ldots,\mathbf a_i$ appearing in any $E\in\mathcal F_i$, and let $\mathcal F_{i+1}$ be the collection of $E\in\mathcal F_i$ containing $\mathbf a_{i+1}$. Then $\mathcal F_c$ will contain only the minimal element $E$ of $\mathcal F$.
\end{proof}

Notice, $E\preceq F$ does not imply $\sum_{\mathbf a\in E}\mathbf a\preceq\sum_{\mathbf b\in F}\mathbf b$. For example, using the standard lexicographic ordering on $\mathbb K[X_0,X_1,X_2]$, we have
\begin{equation}\label{smallersum}
\{(2,0,0),(0,0,2)\}\prec\{(1,1,0),(1,0,1)\},\quad (2,0,2)\succ(2,1,1).
\end{equation}

Define $\tilde D(d)$ to be the hyperplane in $\Q^{n+1}$ that contains $D(d)$. 

\begin{lemma}[Dumnicki]\label{dumnonspeccond}
A well-defined triple $(n,E,(m))$ is special if and only if $E$ is contained in a degree-$(m{-}1)$ hypersurface in $\tilde D(d)$ (i.e. iff there exists a nonzero homogeneous polynomial of degree $m-1$ in $\Q[A_0,\ldots,A_n]$ which vanishes at every $\mathbf a\in E$).
\end{lemma}
\begin{proof}
See Lemma 8 in \cite{MR2342565}.
\end{proof}

Suppose $E\subseteq D(d)$ is contained in a subspace $S$ of $\Q^{n+1}$. Consider $\proj S$ as the projective space of lines in $S$ through the origin, and define $W^S(m-1,E)$ to be the subspace of $H^0(\ocal_{\proj S}(m-1))$ of sections vanishing at all of the points over $E$. When $S$ is all of $\Q^{n+1}$ (the case we will consider most often), we simply write $W^n(m-1,E)$.

Notice that we can identify $W^n(m-1,E)$ with the subspace of homogeneous degree-$(m{-}1)$ polynomials in $\Q[A_0,\ldots,A_n]$ which vanish at every point of $E$. Hence we can rephrase Lemma \ref{dumnonspeccond} as saying $E$ with $\#E={m+n-1\choose n}$ is special if and only if $W^n(m-1,E)=0$. In fact, a closer inspection of the proof we cited in \cite{MR2342565} gives us the following.

\begin{corollary}\label{nonspeccond}
An over- or well-determined triple $(n,E,(m))$ is non-special if and only if
$$\dim W^n(m-1,E)={m+n-1\choose n}-\#E.$$
In particular, if a well- or under-defined triple $(n,F,(m))$ is non-special and $E\subseteq F$, then $(n,E,(m))$ is non-special. 
\end{corollary}

Also, the following lemma is elementary, but we will use it frequently.

\begin{lemma}\label{addpoint}
For any $E\subseteq D(d)$ and $\mathbf a\in D(d)$ we have
$$\dim W^n(m-1,E)\geq\dim W^n(m-1,E\cup\{\mathbf a\})\geq\dim W^n(m-1,E)-1.$$
\end{lemma}
\begin{proof}
Notice that $W^n(m-1,E\cup\{\mathbf a\})$ is the vanishing of a single (possibly zero) linear condition on $W^n(m-1,E)$. Hence, adding a point either reduces the dimension by one or leaves it the same.
\end{proof}

\begin{example}
Consider the subset $E$ of $D(7)$ illustrated in Figure \ref{wspecial}. There is a pencil of quadrics passing through the five points of $E$---the line shown plus any line through the remaining point. Hence $W^2(2,E)=2$, but ${3+2-1\choose 2}-5=1$, and so $(2,E,(3))$ is special.
\begin{figure}
\centering
\includegraphics{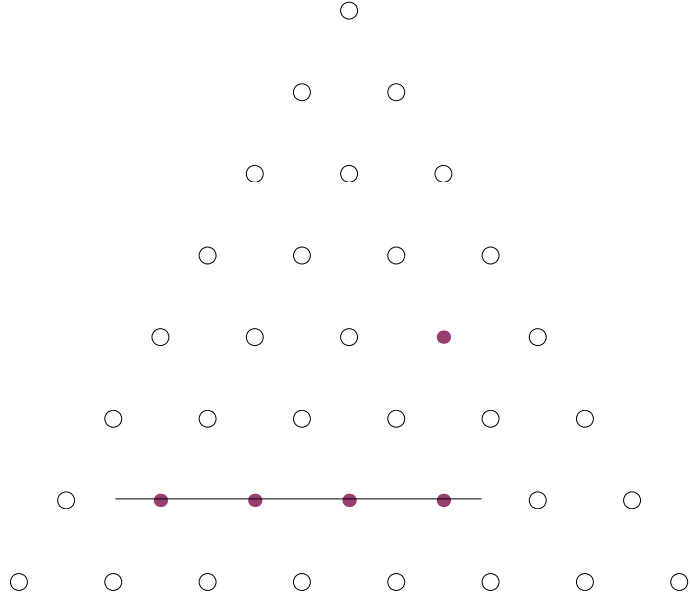}
\caption{A subset $E$ of $D(7)$ for which $(2,E,(3))$ is special. The filled dots represent the points in $E$.}\label{wspecial}
\end{figure}
\end{example}

Define
$$\mathcal F(D,c,m)=\left\{E\in\mathcal E(D,c)\left|\text{$(n,E,(m))$ is non-special}\right.\right\}.$$

The following proposition seems innocuous at first, but in light of examples like (\ref{smallersum}), it should actually be somewhat surprising, and the proof is slightly technical.

\begin{proposition}\label{firsts}
Suppose $c\leq {m+n-1\choose n}$, and let $E$ be the minimal element of $\mathcal F(D,c,m)$ with respect to the $\preceq$-lexicographical ordering for some monomial ordering $\preceq$. Then every $F\in\mathcal F(D,c,m)$ has the property that
$$\sum_{\mathbf a\in E}\mathbf a\preceq\sum_{\mathbf b\in F}\mathbf b$$
with equality holding only if $E=F$.
\end{proposition}
\begin{proof}
Suppose $E=\{\mathbf a_1,\ldots,\mathbf a_c\}$ with $\mathbf a_i\prec\mathbf a_{i+1}$ is the minimal element of $\mathcal F(D,c,m)$, and $F=\{\mathbf b_1,\ldots,\mathbf b_c\}$ with $\mathbf b_i\prec\mathbf b_{i+1}$ has the minimal sum of any element of $\mathcal F(D,c,m)$. By way of contradiction, suppose $E\neq F$.

By Corollary \ref{nonspeccond}, we know
$$\dim W^n(m-1,E)=\dim W^n(m-1,F)=M-c,\quad M={m+n-1\choose n}.$$

Let $i$ be minimal with $\mathbf a_i\neq\mathbf b_i$ so that $\mathbf a_i\prec\mathbf b_i$ (since $E\prec F$). Then $F'=\left(F\smallsetminus\{\mathbf b_i\}\right)\cup\{\mathbf a_i\}$ has a sum strictly smaller than $F$, so $F'\notin\mathcal F(D,c,m)$; that is, $F'$ is special. By Corollary \ref{nonspeccond} and Lemma \ref{addpoint} we have,
$$M-c> \dim W^n(m-1,F')\geq\dim W^n(m-1,F\smallsetminus\{\mathbf b_i\})=M-c+1.$$
Hence $\dim W^n(m-1,F')=M-c+1$

Again using Lemma \ref{addpoint}, we have
$$\dim W^n(m-1,F')-1\leq\dim W^n(m-1,F\cup\{\mathbf a_i\})\leq\dim W^n(m-1,F)$$
and so the middle dimension must be $M-c$. Thus $F\cup\{\mathbf a_i\}$ is special.

Let $F''$ be any minimal (with respect to containment) subset of $F$ with the property that $F''\cup\{a_i\}$ is special. Note that $F''$ cannot contain only $\mathbf b_j$ with $j<i$, else
$$F''\cup\{\mathbf a_i\}\subseteq\{\mathbf a_1,\ldots,\mathbf a_i\}\subseteq E.$$ 
(Remember $E$ is non-special, and by Corollary \ref{nonspeccond}, its subsets are as well). Hence $\left(F''\smallsetminus\{\mathbf b_j\}\right)\cup\{\mathbf a_i\}$ is non-special for some $j\geq i$, which implies $\mathbf b_j\succ\mathbf a_i$.

Consider the following set of properties that some subset $\hat F\subseteq F$ may have:
\begin{equation}\label{hatprops}
\text{$F''\subseteq\hat F\subseteq F$;\quad
$\hat F\cup\{\mathbf a_i\}$ is special;\quad
$(\hat F\smallsetminus\{\mathbf b_j\})\cup\{\mathbf a_i\}$ is non-special.}
\end{equation}
We claim that if $G$ satisfies (\ref{hatprops}), then for any $\mathbf b_k\in F\smallsetminus G$ with $k\neq j$, $G\cup\{\mathbf b_k\}$ satisfies the properties of (\ref{hatprops}) as well. Noticing that $F''$ satisfies (\ref{hatprops}), this will allow us to apply the claim recursively, adding the points of $F\smallsetminus F''$ one at a time, and at the end conclude that $(F\smallsetminus\{\mathbf b_j\})\cup\{\mathbf a_i\}$ is non-special and hence in $\mathcal F(D,c,m)$. But $(F\smallsetminus\{\mathbf b_j\})\cup\{\mathbf a_i\}$ has a sum strictly less than that of $F$, contradicting the minimality assumption and proving the theorem.

To prove the claim, suppose $G$ satisfies (\ref{hatprops}). Start by letting $\mathbf b_k\in F\smallsetminus G$, $k\neq j$. We then have $F''\subseteq G\subseteq F$, and since $F''\subseteq G$ and $F''$ is minimal, we also have $G\cup\{\mathbf a_i\}$ special. So we have only to prove that $\left(G\smallsetminus\{\mathbf b_j\}\right)\cup\{\mathbf a_i,\mathbf b_k\}$ is non-special. If it were special, then we would have, from two applications of Lemma \ref{addpoint},

\begin{eqnarray*}
W^n(m-1,\left(G\smallsetminus\{\mathbf b_j\}\right)\cup\{\mathbf a_i,\mathbf b_k\})&=&W^n(m-1,\left(G\smallsetminus\{\mathbf b_j\}\right)\cup\{\mathbf b_k\})\\
W^n(m-1,\left(G\smallsetminus\{\mathbf b_j\}\right)\cup\{\mathbf a_i,\mathbf b_k\})
&=&W^n(m-1,\left(G\smallsetminus\{\mathbf b_j\}\right)\cup\{\mathbf a_i\})
\end{eqnarray*}

Since the right-hand sides would then equal, we would be able restrict each space to the set of polynomials vanishing at $\mathbf b_j$ and obtain

\begin{equation}\label{ws}
W^n(m-1,G\cup\{\mathbf b_k\})=W^n(m-1,G\cup\{\mathbf a_i\})
\end{equation}

However, we know $G\cup\{\mathbf b_k\}$ is non-special (it is a subset of $F$), and we know $G\cup\{\mathbf a_i\}$ is special (it contains $(G\smallsetminus\{\mathbf b_j\})\cup\{\mathbf a_i\}$ which is special by assumption). Hence the dimensions of the two spaces in (\ref{ws}) are not equal, and we have a contradiction. Therefore we must have had $\left(G\smallsetminus\{\mathbf b_j\}\right)\cup\{\mathbf a_i,\mathbf b_k\}$ non-special.
\end{proof}

As a corollary of this proposition, we obtain the following theorem, which will be one of our main tools for constructing reductions. In fact, we can think of the reductions constructed by Dumnicki in \cite{MR2289179} (by diagram-cutting) and \cite{MR2325918} (see Remark \ref{redremark})  as applications of Theorem \ref{corfirst}.

\begin{theorem}\label{corfirst}
For any triple $(n,D,\mathbf m)$, if $\mathcal F(D,c,m_r)$ is non-empty for some $c>0$, then its minimal element with respect to any monomial ordering $\preceq$ is a reduction.
\end{theorem}

\begin{proof}
Let $E$ be the minimal element of $\mathcal F(D,c,m_r)$. Then by Proposition \ref{firsts}, we know that for any $F\subseteq D$ with $\#F=\#E$ and $F\neq E$, 
$$\mathbf a_r(\varnothing,\ldots,\varnothing,F)\succneqq\mathbf a_r(\varnothing,\ldots,\varnothing,E).$$
\end{proof}

One useful feature of Theorem \ref{corfirst} is that once we know $\mathcal F(D,c,m_r)$ is non-empty, we may choose any monomial ordering and obtain a reduction. In particular, if we are building a reduction algorithm, we may use a different monomial ordering in each step. We can capitalize on this idea with the following algorithm aimed at proving non-speciality.

\begin{algorithm}\label{algorithm0}~

INPUT: A triple $(n,D,\mathbf m)$, and an ordered $r$-tuple of monomial orderings $(\preceq_1,\ldots,\preceq_r)$.

OUTPUT: A lower bound on $\dim V^n_D(\mathbf m)$ and either ``non-special'' or ``undecided''.

ALGORITHM:
Define $D_r=D$. Recursively take the largest $c_i\leq{m_i+n-1\choose n}$ for which $\mathcal F_i=\mathcal F(D_i,c_i,m_i)$ is non-empty, let $E_i$ be the minimial element of $\mathcal F_i$, and define $D_{i-1}=D_i\smallsetminus E_i$. Output $\dim V^n_D(\mathbf m)\leq\#D_0$. If $\#D_0=\edim(n,D,\mathbf m)$, then output ``non-special''. Otherwise, output ``undecided''.
\end{algorithm}

Notice that Algorithm \ref{algorithm0} will either prove the non-speciality a triple, or it will come out inconclusive---it cannot prove speciality.

One of the drawbacks of applying Algorithm \ref{algorithm0} is that computing the minimal element of $\mathcal F_i$ may be quite difficult. In particular, the somewhat na\"ive approach of using Corollary \ref{nonspeccond} to test each element of $\mathcal E(D_i,c_i)$ (in the order determined by the well-ordering) for speciality until the minimal $E_i\in\mathcal F_i$ is found (or until $\mathcal F_i$ is found to be empty) may be very computation-heavy. 

However, the following generalization of a lemma of Dumnicki allows us to obviate the linear algebra test for speciality for a large class of examples.

\begin{lemma}\label{anynaf}
Suppose $(n,E,(m))$ is an over- or well-determined triple. For some $k\leq m$, let $H_1,\ldots,H_k$ be distinct dimension-$n$ subspaces of $\Q^{n+1}$ which intersect $\tilde D(d)$ in parallel hypersurfaces of $\tilde D(d)$. Let $E'=E\cap(H_1\cup\cdots\cup H_k)$.
\begin{enumerate}
\item\label{anynno1} Suppose
$$\#E'>\sum_{i=1}^k{m-i+n-1\choose n-1}={m+n-1\choose n}-{m-k+n-1\choose n}.$$
Then $(n,E,(m))$ is special.
\item\label{anynyes} Suppose that $E=E'$, and that for $i=1,\ldots,k$
\begin{equation}\label{gennonspeccond}
\#E\cap H_i\leq{m-i+n-1\choose n-1}.
\end{equation}
If for each $i$,
\begin{equation}\label{gennonspeccond2}
\dim W^{H_i}(m-i+1,E\cap H_i)={m-i+n-1\choose n-1}-\#(E\cap H_i),
\end{equation}
then $(n,E,(m))$ is non-special.
\end{enumerate}
\end{lemma}
\begin{proof}~
\begin{enumerate}
\item First notice that since $\#E\leq{m+n-1\choose n}$, the assumption implies $k<m$. Now, there exists a nontrivial $f\in W^n(k,E')$ vanishing on $H_1\cup\cdots\cup H_k$. Hence $f$ vanishes on $E'$, which implies by Corollary \ref{nonspeccond} that $(n,E',(m))$ is special. Then since $E\supseteq E'$, $E$ must be special as well.
\item By Corollary \ref{nonspeccond}, it suffices to prove this for the case where $k=m$ and $\#E={m+n-1\choose n}$. In this case, we necessarily have equality in (\ref{gennonspeccond}) for all $i$.

Suppose, by way of contradiction, there exists a nontrivial $f\in W^n(m-1,E)$, and let $S$ be its vanishing set. Let $1\leq i<m$ and assume $S\supset H_j$ for all $j<i$ (vacuously if $i=1$). Then $S$ consists of the union of all $H_j$ with $j<i$ together with a degree $m-i+1$ hypersurface. Therefore, since the $H_j$ are disjoint, either $S\cap H_i$ has degree $m-i+1$ or it is all of $H_i$. However, by our assumption, (\ref{gennonspeccond2}) reduces to $W^{H_i}(m-i+1,E\cap H_i)=0$, so we see that the former possibility is prohibited. Hence $S\supset H_i$. Therefore, by induction $S$ contains $H_1,\ldots,H_{m-1}$, but since $S$ has degree $m-1$, their union must be all of $S$. But $S$ was also supposed to contain the single point in $H_m$, and hence we have a contradiction.
\end{enumerate}
\end{proof}

\begin{example}
In Figure \ref{specialfig}, we illustrate a subset $E$ of $D(7)$. Notice that $\#E'=19$ but ${7\choose2}-{3\choose2}=18$. Hence by  Condition \ref{anynno1} of Lemma \ref{anynaf}, $(2,E,(6))$ is special.
\begin{figure}
\centering
\includegraphics{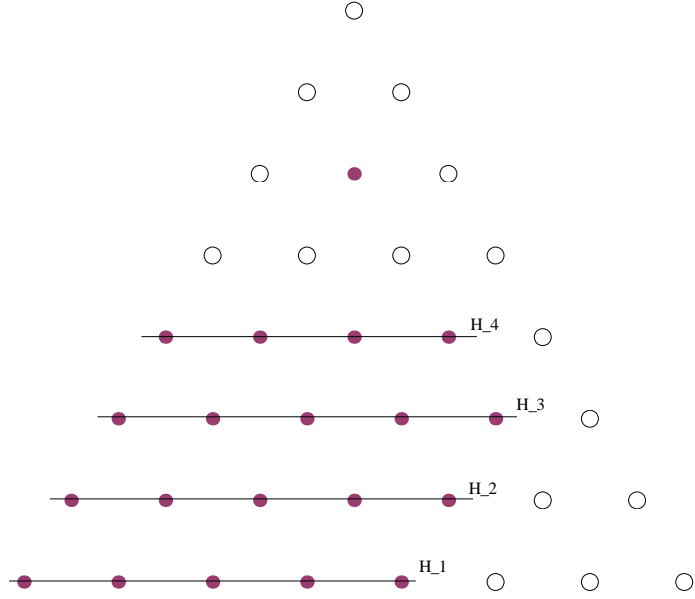}
\caption{A subset $E$ of $D(7)$ so that $(2,E,(6))$ is special by Condition \ref{anynno1} of Lemma \ref{anynaf}. The intersections $H_i\cap \tilde D(7)$ are illustated for clarity.}\label{specialfig}
\end{figure}
\end{example}

One powerful aspect of Lemma \ref{anynaf} to notice is that using Condition \ref{anynyes}, we can use our knowledge of non-speciality in low dimensions to determine non-speciality in higher dimensions. This arises from the fact that we can re-phrase (\ref{gennonspeccond2}) as
$$\text{$E\cap H_i$ is non-special as a subset of $D^{n-1}(d')\subset H^i\cong\Q^n$.}$$

First we notice that the case $n=1$ is trivial: $(1,E,(m))$ is always non-special. This follows directly from Corollary \ref{nonspeccond}. From here, we apply Lemma \ref{anynaf} in two ways: first by applying Condition \ref{anynyes} to construct a large class of non-special $(n,E,(m))$ for general $n$; then by describing the case where $n=2$ as thoroughly as possible so that we can apply it to specific examples effectively.

\begin{definition}
Define a \emph{scrambled $1$-simplex of size $m$} to be any set of $m$ colinear points in $D(d)$.

Then, recursively define a \emph{scrambled $k$-simplex of size $m$} to be a set $E\subseteq D(d)$ so that there exist $m$ distinct dimension-$(k{-}1)$ subspaces, $H_1,\ldots,H_m$ of $\Q^{n+1}$ whose intersections with $\tilde D(d)$ are parallel, such that $E\subseteq H_1\cup\cdots\cup H_m$, and $E\cap H_i$ is a scrambled $(k{-}1)$-simplex of size $i$. 
\end{definition}

\begin{example}\label{skew2ex}
In Figure \ref{skew2fig}, we illustrate a scrambled $2$-simplex of size $4$. Notice that $E\cap H_i$ is a set of $i$ colinear points for $i=1,2,3,4$.
\begin{figure}
\centering
\includegraphics{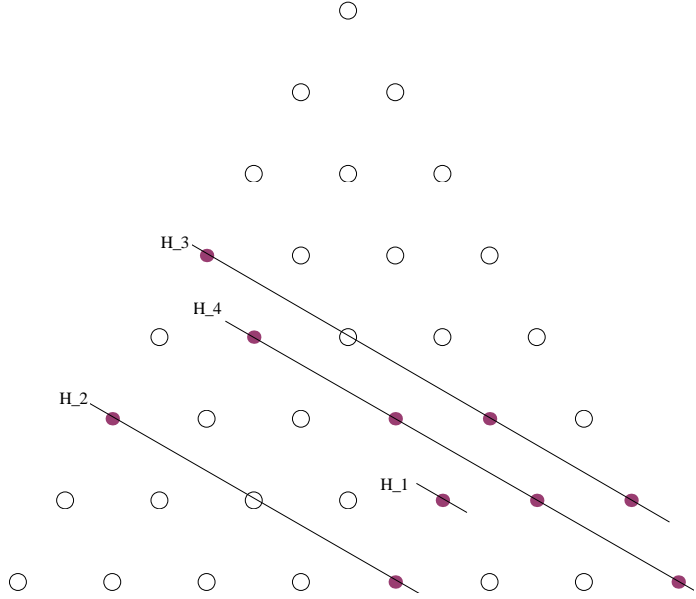}
\caption{A scrambled $2$-simplex of size $4$ in $D(7)$. The intersections $H_i\cap\tilde D(7)$ are illustrated as well for clarity.}\label{skew2fig}
\end{figure}
\end{example}

\begin{proposition}\label{skewsimplex}
Any subset $E$ of a scrambled $n$-simplex of size $m$ has the property that $(n,E,(m))$ is non-special.
\end{proposition}
\begin{proof}
Use Condition \ref{anynyes} of Lemma \ref{anynaf} and induction on $n$.
\end{proof}

We now specialize to the case of $n=2$. In this case, we define a \emph{row} in $D\subseteq D(d)$ to be the (possibly empty) intersection of $D$ with a dimension-$2$ subspace of $\Q^3$.

\begin{proposition}\label{af}
Suppose $(2,E,(m))$ is an over- or well-determined triple.
\begin{enumerate}
\item\label{no} If there are $k<m$ parallel rows $R_1,\ldots,R_k$ so that 
$$\#(R_1\cup\cdots\cup R_k)\cap E>{m+1\choose 2}-{m+1-k\choose2},$$
then $(2,E,(m))$ is special.
\item\label{yes} If $E$ is contained in the union of $m$ parallel rows $R_1,\ldots,R_m$, so that $\#R_i\cap E\leq i$ for all $i$, then $(2,E,(m))$ is non-special.
\end{enumerate}
\end{proposition}
\begin{proof}~
\begin{enumerate}
\item Apply Condition \ref{anynno1} of Lemma \ref{anynaf}.
\item This condition is equivalent to saying $E$ is a subset of a scrambled $2$-simplex of size $m$. Apply Proposition \ref{skewsimplex}.
\end{enumerate}
\end{proof}

Our goal now is to apply Proposition \ref{af} in such a way as to completely avoid the (possibly computation-heavy) linear algebra test of Corollary \ref{nonspeccond}. Roughly speaking, for a given monomial ordering $\preceq$, our strategy for $n=2$ will be to take the largest $c$ we can think of for which we can determine the minimal element of $\mathcal F(D,c,m)$ using only Proposition \ref{af}. That is, take the largest $c$ for which we can show that the minimal $E\in\mathcal E(D,c)$ satisfying Condition \ref{yes} in Corollary \ref{af} is greater only than elements of $\mathcal E(D,c)$ which satisfy Condition \ref{no}. This will show \emph{a fortiori} that $E$ is minimal in $\mathcal F(D,c,m)$.

Let $\preceq$ be a lexicographic (resp. reverse lexicographic) monomial ordering on $\mathbb K[X_0,X_1,X_2]$, say with the convention that $X_{i_0}\prec X_{i_1}\prec X_{i_2}$. Now we define a \emph{$\preceq$-row} in $D$ to be a row of the form $R(k)=\{(a_0,a_1,a_2)\in D|a_{i_0}=k\}$. Then we can define an ordering of $\preceq$-rows
$$R(k)\succeq R(l)\Longleftrightarrow k\leq l\;\text{(resp. $k\geq l$)}.$$
Now for any two $\preceq$-rows $R_i$ and $R_j$, if $\mathbf a\in R_i$, $\mathbf b\in R_j$ and $R_i\succ R_j$, then $\mathbf a\succ\mathbf b$.

We can now present a generalization of Dumnicki and Jarnicki's notion of ``weak $m$-reduction'' from  \cite{MR2325918}.

\begin{definition}\label{reddef}
Let $\preceq$ be a lexicographic or reverse lexicographic monomial ordering on $\mathbb K[X_0,X_1,X_2]$ with the convention
$$X_{i_0}\prec X_{i_1}\prec X_{i_2}.$$
Suppose $D\subseteq D(d)$. Suppose there are $\bar k$ non-empty $\preceq$-rows in $D$. Then let $k=\min\left\{m,\bar k\right\}$, and name the minimal $k$ $\preceq$-rows in $D$
$$R_1\prec\cdots\prec R_k.$$
Let $\Omega_1=\{1,\ldots,m\}$. Recursively define for $1\leq j\leq k$,
$$u_j=\min\{\max\Omega_j,\#R_j\},\quad u'_j=\min\{s\in\Omega_j|s\geq u_j\},\quad \Omega_{j+1}=\Omega_j\smallsetminus\{u'_j\}.$$
Then define the \emph{$(m,\preceq)$-reduction of $D$} to be
$$\red_\preceq(D,m)=\bigcup_{j=1}^m\{\text{minimal $u_j$ elements of $R_j$}\}.$$
\end{definition}

\begin{example}\label{reductionex}
Figure \ref{reductionfig} shows the $(5,\preceq)$-reduction of a subset $D$ of $D(7)$, where $\preceq$ is the reverse lexicographic ordering with $X_1\prec X_0\prec X_2$. Here is the step-by-step construction:
$$\begin{array}{c|c|c|c}
i&\Omega_i&u_i&u_i'\\
\hline
1&\{1,2,3,4,5\}&3&3\\
2&\{1,2,4,5\}&3&4\\
3&\{1,2,5\}&4&5\\
4&\{1,2\}&2&2\\
5&\{1\}&1&1
\end{array}$$

\begin{figure}
\centering
\includegraphics{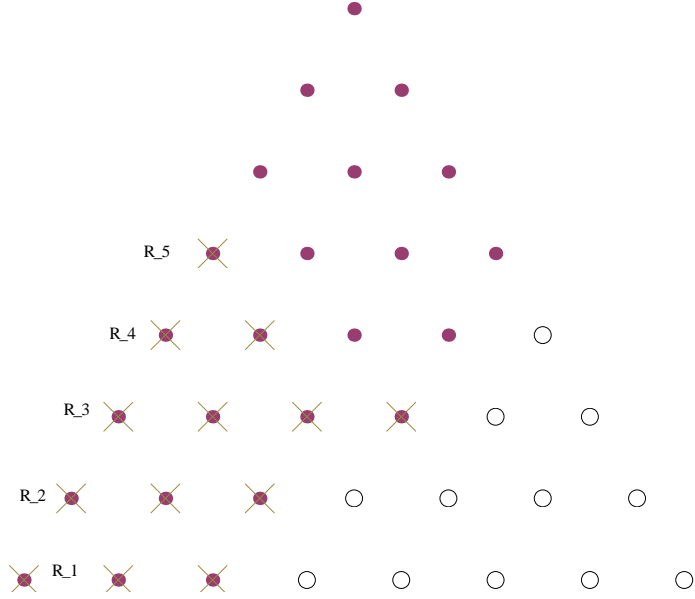}
\caption{The $(5,\preceq)$-reduction of a subset $D$ of $D(7)$, where $\preceq$ is the reverse lexicographic ordering with $X_1\prec X_0\prec X_2$. The points of $D$ are denoted by solid dots, and the points of $D'=\red_\preceq(D,5)$ are denoted by $\times$'s.}\label{reductionfig}
\end{figure}
\end{example}

The following Lemma justifies the re-use of the word ``reduction''.

\begin{lemma}\label{redlemma}
Given a triple $(2,D,\mathbf m)$,  let $D'=\red_\preceq(D,m_r)$. Then for some $G\subseteq D(d)$ containing $D$, $D'\cup (G\smallsetminus D)$ is a reduction for $(n,G,\mathbf m)$. 
\end{lemma}

\begin{remark}\label{redremark}
In \cite{MR2325918}, only one monomial ordering is considered. The details of the proof of the corresponding lemma from that paper are given in \cite{MR2342565}. The ``bean-counting'' aspect of our proof is mostly the same, but in the end we appeal to Theorem \ref{corfirst} to prove that we end up with a reduction.
\end{remark}

\begin{proof}
First, we prove the case where $\preceq$ is a lexicographic monomial order, and then note how the proof must be modified to cover the reverse lexicographic case. We use the notation of Definition \ref{reddef}. We assume without loss of generality that $\preceq$ is such that $X_0\prec X_1\prec X_2$.

For each row $R_i$, define $R_i'$ to be the row of $D(d)$ containing $R_i$. Notice that $R'_i=\{(a_0,a_1,a_2)\in D(d)|a_0=\alpha_i\}$ for some 
$$0\leq \alpha_k<\cdots<\alpha_1\leq d.$$
We note that $\#R'_i=d-\alpha_i+1\geq i$.

Notice that
$$u_j'\leq u_j+j-1\leq\max\Omega_j+j-1\leq\#R_1+j-1\leq\#R_1'+j-1\leq \#R_j'.$$
The first inequality is true because $\Omega_j$ must contain one of $u_j,u_j+1,\ldots,u_j+j-1$. Define
$$G=D\cup\bigcup_{j=1}^k\{\text{an additional $u'_j-u_j$ points from $R_j'$}\}.$$

Now, let $E=\red_\preceq(D,m)\cup(G\smallsetminus D)$. The claim is then that $E$ is the minimal element of $\mathcal F(G,\#E,m)$. First notice that $\#(E\cap R_j')=u'_j$ since either 
\begin{enumerate}
\item $u_j=\max\Omega_j$ so that $u_j=u_j'$ and then $G\cap R_j'=R_j$, or
\item $u_j=\#R_j$ so that $u_j'\geq\#R_j$ and $G\cap R_j'=E\cap R_j'=R_j\cup\{\text{$u_j'-\#R_j$ points}\}$.
\end{enumerate}
Since the $u_j'$ are by definition distinct integers no more than $m$, we know that $E$ is a subset of a scrambled $2$-simplex of size $m$. Hence $E\in\mathcal F(G,\#E,m)$.

Suppose that $F$ with $\#F=\#E$ has $F\prec E$. Assume $E=\{\mathbf a_1,\ldots,\mathbf a_c\}$ with $\mathbf a_i\prec\mathbf a_{i+1}$ and $F=\{\mathbf b_1,\ldots,\mathbf b_c\}$ with $\mathbf b_i\prec\mathbf b_{i+1}$. Let $i$ be minimal so that $\mathbf b_i\neq\mathbf a_i$, which implies $\mathbf b_i\prec\mathbf a_i$. 

Say $\mathbf a_i\in R'_j$. Then we know $\mathbf b_i$ is not in $R_j'$ because $E$ contains the minimal $u_j'$ elements of $G\cap R'_j$. Now $\mathbf a_{i-1}$ must be in either $R_j'$ or $R_{j-1}'$, but since $\mathbf b_i\succ\mathbf b_{i-1}=\mathbf a_{i-1}$ and $\mathbf b_i\notin R_j'$, we must have $\mathbf b_{i-1},\mathbf b_i\in R_{j-1}'$. In fact,
$$R_{j-1}'\ni\mathbf a_{l-u_{j-1}'}=\mathbf b_{l-u_{j-1}'}\prec\cdots\prec\mathbf b_l.$$
This shows that $F\cap R_{j-1}'$ contains more than $u'_{j-1}$ elements.

Now, because of this, we must have $\#F\cap R_{j-1}'>\omega$, where $\omega=\max\Omega_{j-1}$. Since $\Omega_{j-1}$ does not contain $\omega+1,\ldots,m$, there must have been $m-\omega$ rows preceding $R_{j-1}'$ whose intersections with $F$ contain $\omega+1, \ldots, m$ elements respectively. But then the number of elements in the union of these rows together with $R_{j-1}'$ intersected with $F$ is more than
$$\omega+(\omega+1)+\ldots+m={m+1\choose2}-{\omega\choose2}.$$
Hence by Proposition \ref{af}, Condition \ref{no}, $F$ is in fact special, showing that $E$ is in fact minimal in $\mathcal F(G,\#E,m)$. Therefore we may apply Theorem \ref{corfirst}, and so $E$ is a reduction.

In the reverse lexicographic case, we do not necessarily have $\#R_i'\geq i$. This requires us to choose $G$ more carefully, but the steps are the same as in the lexicographic case. With different notation, Dumnicki proves this case in \cite{MR2342565}.
\end{proof}

\begin{example}
\begin{figure}
\centering
\includegraphics{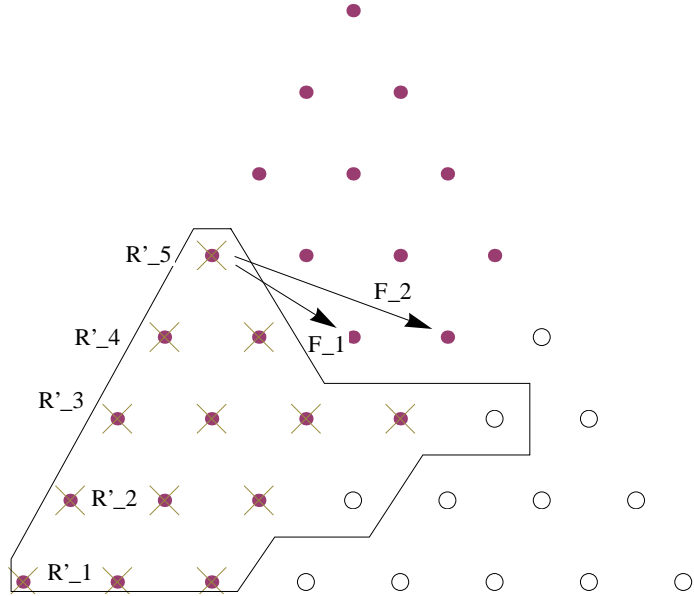}
\caption{The same illutration from Figure \ref{reductionfig} with additional labels. $D$ and $D'$ are as above. $E$ is the set of points inside the polygon, and $G$ is $E\cup D$. $F_i$ is obtained by taking $E$ and exchanging the point at the tail of the arrow labeled $F_i$ for the point at the head ($i=1,2$).}\label{augmentedfig}
\end{figure}
We return to Example \ref{reductionex} to demonstrate the proof. First, notice that $u_2'-u_2=1$, $u_3'-u_3=1$, and $u_j'=u_j$ for $j=1,4,5$. Hence $G$ is defined to be $D$ plus one extra point from each of $R_2'$ and $R_3'$ as shown in Figure \ref{augmentedfig}. It does not matter which extra points are chosen, but we know that there are enough to choose from.

Notice that $E=D'\cup (G\smallsetminus D)$ is a scrambled $2$-simplex of size $5$. Hence $E\in\mathcal F(G,15,5)$.

There are only two possible $F\subseteq G$ with $\#F=15$ and $F\prec E$; call them $F_1$ and $F_2$, as illustrated. In both cases, $F_i\cap R_4'>u_4'=2$, which necessarily means that previous rows of $F_i$ must have contained $3$, $4$, and $5$ elements---this is clearly the case. And because of this fact, $\#F_i\cap(R_1'\cup R_2'\cup R_3'\cup R_4')=15$, showing that $F_i\notin\mathcal F(G,15,5)$ for $i=1,2$ by Proposition \ref{af}. Therefore $E$ is a reduction of $G$.
\end{example}

In light of Lemma \ref{redlemma}, we are justified in writing the following algorithm.

\begin{algorithm}\label{algorithm1}~

INPUT: $d\in\N$, $\mathbf m\in\N^r$, an ordered $r$-tuple of lexicographic or reverse lexicographic monomial orderings $(\preceq_1,\ldots,\preceq_r)$.

OUTPUT: A lower bound on $\dim V^2(d,\mathbf m)$, and either ``non-special'' or ``undecided''.

ALGORITHM:
Define $D_r=D(d)$. Then recursively let $D_{i-1}=D_i\smallsetminus\red_{\preceq_i}(D_i,m_i)$ for $i=m,\ldots,1$. Output $\dim V^2(d,\mathbf m)\leq \#D_0$. If $\#D_0=\edim(2,D(d),\mathbf m)$, then output ``non-special''. Otherwise, output ``undecided''.
\end{algorithm}

\begin{remark}
Compared with Algorithm \ref{algorithm1}, Algorithm \ref{algorithm0} is more likely to detect the non-speciality of a triple because Algorithm \ref{algorithm1} does not necessarily maximize the number of elements in each reduction. However, Algorithm \ref{algorithm1} is much cheaper computationally. It requires no linear algebra calculations which appeal to Corollary \ref{nonspeccond}; only the combinatoric comparisons necessary to define the $(m,\preceq)$-reductions are needed. For calculations with $|\mathbf m|>100$, Algorithm \ref{algorithm0} becomes impractical.

We also note that Lemma \ref{anynaf} can be used to produce higher-dimensional analogues of Definition \ref{reddef} leading to higher-dimensional analogues of \ref{algorithm1}.

Finally, we remark this algorithm could likely be improved by incorporating other techniques such as Cremona transformations, as in the algorithms developed in \cite{MR2325918}. We forgo the use of other methods for simplicity and to highlight the power of Lemma \ref{redlemma} on its own (see for example the results in Section \ref{nagatasubsection}).
\end{remark}

Notice that there are $12$ possible monomial orderings that we can use for each $\preceq_i$ in the input of the above algorithm. This means there are $12^r$ possible $r$-tuples of monomial ordering we could potentially test.

We will use the notation $\lex(i_0,i_1,i_2)$ to denote the lexicographic ordering with $X_{i_0}\prec X_{i_1}\prec X_{i_2}$. We also denote by $\rlex(i_0,i_1,i_2)$ the reverse lexicographic ordering with $X_{i_2}\prec X_{i_1}\prec X_{i_0}$.

\begin{example}\label{alg1ex}
The reduction algorithm for $(2,D(7),3^{\times6})$ in Figure \ref{redpic} can be produced using Algorithm \ref{algorithm1} by using the sextuple of monomial orderings
$$\big(\lex(1,2,0),\lex(1,2,0),\lex(1,2,0),\lex(0,1,2),\rlex(1,2,0),\rlex(1,2,0)\big).$$

\begin{figure}
\centering
\includegraphics{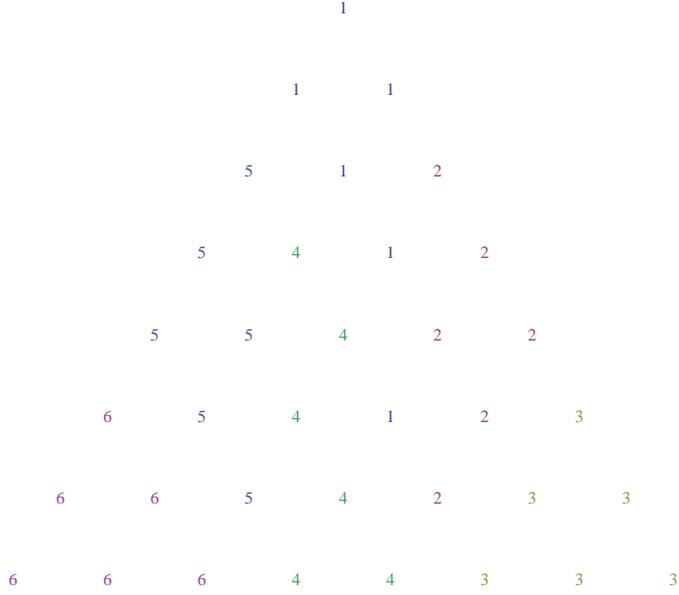}
\caption{A full reduction algorithm for $(2,D(7),3^{\times 6})$. $D_i\smallsetminus D_{i-1}$ consists of the points labeled by $i$ for $i=1,2,3,4,5,6$.}\label{redpic}
\end{figure}
\end{example}

\section{Application of Algorithms}\label{results}

We will apply Theorem \ref{genredalgtheorem} to examples stemming from two different areas of study. First, we produce new bounds on the multi-point Seshadri constants of $\proj^2$. Second, we recover a result from Evain \cite{MR2125451} which generalizes the known cases of Nagata's Conjecture to higher dimensions, proving the emptiness $\left(n,D(d),m^{\times s^n}\right)$ when $d\leq ms$ (with a few well-known exceptions).

\subsection{Bounding multi-point Seshadri constants of $\proj^2$.}\label{nagatasubsection}

\begin{definition}[See for example \cite{MR2555949}]
Let $(X,L)$ be a smooth polarized variety, $p_1,\ldots,p_r\in X$. Then we define the \emph{multi-point Seshadri constant of $L$ at $p_1,\ldots,p_r$} to be
$$\epsilon(L;p_1,\ldots,p_r):=\inf_{\text{curves $C$}}\frac{L.C}{\sum_{i=1}^r\mult_{p_i}C}.$$
If the points are taken to be very general and $(X,L)$ is understood, we simply write $\epsilon(r)$.
\end{definition}

Applying this to the polarized variety $(\proj^2,\ocal_{\proj^2}(1))$, we can equate curves with nonzero sections of $\ocal_{\proj^2}(d)$ for some $d$, up to nonzero scalar multiples. Then we can write
$$\epsilon(r)=\inf_{\substack{d\geq1,\mathbf m\in\N^r\\0\neq V^2(d,\mathbf m)}}\frac{d}{m_1+\cdots+m_r}.$$

In this language, Nagata's Conjecture states that for $r\geq9$, $\epsilon(r)=\frac 1{\sqrt r}$. That $\epsilon(r)$ is no more than $\frac 1{\sqrt r}$ can be proved from first principles (see for example \cite{MR2555949,MR2098342}), but only for $r$ a perfect square has equality been shown---in fact by Nagata in \cite{MR0088034}. 

To show that $\epsilon(r)\geq e$ for some constant $e$, one must prove that $V^2(d,\mathbf m)=0$ whenever $\frac{d}{m_1+\cdots+m_r}\leq e$ . One can check that for $r>9$, the condition $\frac{d}{m_1+\cdots+m_r}\leq e\leq\frac 1{\sqrt r}$ implies $(2,D(d),\mathbf m)$ is over-defined, and so showing $V^2(d,\mathbf m)=0$ is equivalent to showing that $(2,D(d),\mathbf m)$ is non-special.

In \cite{MR2574368}, Harbourne and Ro\'e construct, for each non-square $r>9$, an increasing sequence of rational numbers limiting to $\frac{1}{\sqrt r}$ with no other accumulation points, so that $\epsilon(r)$ must either be one of the values in that sequence or $\frac{1}{\sqrt r}$. Furthermore, they show that to rule out any one of the rational values, it suffices to show the non-speciality of a finite number of ``candidate'' triples (we are using the word ``triple'' differently here than in \cite{MR2574368}). Hence, if enough candidate triples are shown to be non-special, one can produce a lower bound arbitrarily close to the conjectured value of $\frac 1{\sqrt r}$. If a candidate triple is special, then it is a counter-example to Nagata's Conjecture.

\begin{example}
\begin{figure}
\centering
\includegraphics{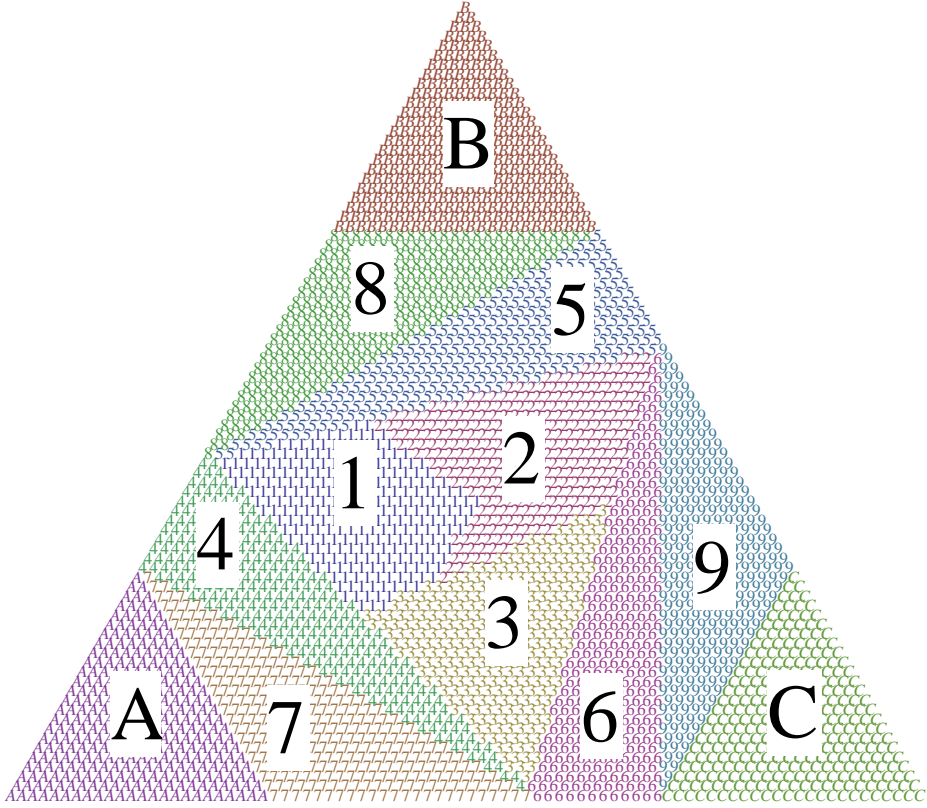}
\caption{A full reduction algorithm for $(2,D(83),24^{\times12})$. To denote points in $D_i\smallsetminus D_{i-1}$, we use $i$'s for $i=1,\ldots,9$, and we use $A$'s, $B$'s and $C$'s for $i=10,11,12$ respectively.}\label{12casefig}
\end{figure}
$(2,83,24^{\times 12})$ is one of the candidate triples of which it is necessary to show non-speciality to prove $\epsilon(12)\neq\frac{83}{288}$. We do so by showing that 
$$\dim V^2(83,24^{\times 12})=\edim(2,D(83),24^{\times 12})=0.$$
We use Algorithm \ref{algorithm1} with the following orderings: 
\begin{eqnarray*}
\big(\lex(0,1,2),\lex(1,2,0),\lex(2,0,1),\lex(0,1,2),\lex(1,2,0),\lex(2,0,1),\;\\
\;\;\lex(0,1,2),\lex(1,2,0),\lex(2,0,1),\lex(0,2,1),\lex(1,0,2),\lex(1,0,2)\big).
\end{eqnarray*}
For $i=12,\ldots,2$, $\#(D_i\smallsetminus D_{i-1})={25\choose2}=300$, which is the maximum possible. $\#D_1=270$ and $D_0=\varnothing$. Hence, we obtain a reduction algorithm as pictured in Figure \ref{12casefig}.
\end{example}

While the above example was computed by hand, we used a Mathematica program to systematically perform Algorithm \ref{algorithm1} on a number of candidate triples using a random collection of $r$-tuples of monomial orderings. We summarize our results in Figure \ref{lowerbounds}. As a measure of how close a bound $e$ is to the conjectured value $\frac{1}{\sqrt r}$ of $\epsilon(r)$, we use the $f$-value of our bound (as in \cite{MR2574368}) defined by
$$e=\frac{1}{\sqrt r}\sqrt{1-\frac{1}{f}}.$$
A larger $f$-value corresponds to a better bound. In most cases tested, we were able to quickly produce the best known bounds on $\epsilon(r)$.

\begin{figure}
$$\begin{array}{|c|c|c|cc|c|}
\hline
&\multicolumn{2}{|c|}{\text{Using Algorithm \ref{algorithm1}}}&\multicolumn{3}{|c|}{\text{Previous best known}}\\
\cline{2-6}
r&\epsilon(r)\geq&f(r)\geq&\multicolumn{2}{|c|}{\epsilon(r)\geq}&f(r)\geq\\
\hline
10&60/19 & 361& 117/370&\cite{MR2738381}&13690\\
11&169/561&572.22&106/352&\cite{MR2574368}&402.28\\
12&277/960& 1081.69&83/288&\cite{MR2574368}&300.52 \\
13&191/689&1014.36&90/325&\cite{MR2574368}&325 \\
14&187/700& 1129.03&86/322&\cite{MR2574368}&740.6 \\
15&484/1875&1969.54&426/1651&\cite{MR2543429}&744.55 \\
17&305/1258&1389.43 &136/561&\cite{MR1687571}&1089 \\
18&509/2160& 2178.15&89/378&\cite{MR2574368}&466.94 \\
19&584/2546& 3158.93&170/741&\cite{MR1687571}&28900 \\
20&948/4240&5107.27&1617/7235&\cite{MR2543429}& 1017.5\\
21&559/2562& 3765.83&142/620&\cite{MR2574368}&660.64 \\
22&1074/5038& 5104.88&197/924&\cite{MR1687571}&38809 \\
23&820/3933& 4703.1&115/552&\cite{MR2574368}&576 \\
24&578/2832& 3632.35&8092/39657&\cite{MR2543429}& 1371.71\\
26&673/3432& 4768.67&260/1326&\cite{MR1687571}&2601 \\
27&239/1242& 5193.82&161/837&\cite{MR2574368}&997.96 \\
28&582/3080& 4457.89&201/1064&\cite{MR2574368}&1304.25 \\
29&350/1885& 4901&113/609&\cite{MR2574368}&639.45 \\
30&586/3210&4641.49& 219/1200&\cite{MR2574368}&2130.76 \\
31&746/4154&4638.63 &128/713&\cite{MR2574368}&1093.26 \\
32&724/4096&4681.14 &147/832&\cite{MR2574368}&940.52 \\
33&471/2706&4350.82 &178/1023&\cite{MR2574368}&1093.55 \\
34&653/3808&4902.25 &239/1394&\cite{MR2574368}&1731.93 \\
35&840/4970&5041 &136/805&\cite{MR2574368}&974.47 \\
37&444/2701&5329 &444/2701&\cite{MR1687571}&5329 \\
38&715/4408&4964.35 &265/1634&\cite{MR2574368}&1898.97 \\
39&843/5265&5641.07 &231/1443&\cite{MR2574368}&1779.7 \\
40&449/2840&5170.26 &196/1240&\cite{MR2574368}&1601.66 \\
41&493/3157&6077.22 &160/1025&\cite{MR1989646,MR2574368}&1025 \\
42&985/6384&6785.79 &149/966&\cite{MR2574368}&1306.94 \\
43&1036/6794&6881.1 &236/1548&\cite{MR2574368}&1741.5 \\
44&650/4312&5560.21 &252/1672&\cite{MR2574368}&1985.5 \\
45&872/5850&6556.03& 275/1845&\cite{MR2574368}&3782.25 \\
46&746/5060&6626.19 &217/1472&\cite{MR2574368}&3140.26 \\
47&473/3243&5888.61&994/6815&\cite{MR2574368}&7109.17 \\
48&575/3984&7035.57 &187/1296&\cite{MR2574368}&1521.39 \\
50&601/4250&7372.45&700/4950&\cite{MR1687571}&9801 \\
\hline
\end{array}$$
\caption{Lower bounds on multi-point Seshadri constants of $\proj^2$ found using Algorithm \ref{algorithm1} and the previously best known lower bounds with citations. Notice that our bounds are equal to or better than the previous best for all $r$ except $r=10,19,22,47,50$}\label{lowerbounds}
\end{figure}

\subsection{Non-speciality of $\left(n,D(d),m^{\times s^n}\right)$}

Here we indicate how our methods may be used to recover a theorem of Evain confirming and strengthening certain cases of Iarrobino's Conjecture from \cite{MR1337187}. 

Iarrobino's conjecture states that, apart from a few known counter-examples, if $d^n<rm^n$, then $(n,D(d),m^{\times r})$ is non-special. Notice that the $2$-dimensional case is Nagata's Conjecture. Theorem \ref{nagatagen}, originally proved by Evain in \cite{MR2125451}, confirms the conjecture, in fact with weak inequality, for the case where $r=s^n$ for some $s\in\N$.

\begin{theorem}\label{nagatagen}
Suppose $n\geq2$, $s\geq\max\{2,6-n\}$, $m\geq1$, and $d\leq ms$. Then $\left(n,D(d),m^{\times s^n}\right)$ is non-special. In particular, the linear series of degree-$d$ hypersurfaces in $\proj^n$ based at $s^n$ general points with multiplicity at least $m$ is empty.
\end{theorem}

We start with the case where $d$ is strictly less than $ms$, and then sketch how to extend this to to the case of equality---first in some base cases, and then by induction on $n$.

\begin{proposition}\label{strictprop}
If $n\geq1$, $s\geq1$, $m\geq1$, and $d<ms$, then $\left(n,D(d),m^{\times s^n}\right)$ admits a fully exceptional partition. Moreover, each part of this partition is a subset of a scrambled $n$-simplex of size $m$.
\end{proposition}

\begin{proof}[Sketch of Proof]
Choose some irrational $0<\delta<m-\frac{d}{s}$, and let $\mu=\frac{d}{s}+\delta$. Then consider hyperplanes in $\tilde D(d)$ of the form
$$H_{c,i,j}=\left\{\mathbf a\in\tilde D(d)\left|\sum_{k=i+1}^j a_k=c\mu\right.\right\},\text{ for integers $1\leq c<s$, $0\leq i<j\leq n$}.$$
These hyperplanes will partition $D(d)$ into $s^n$ parts as $\mathbf E=(E_1,\ldots,E_{s^n})$. Because $\mu$ is irrational, no integral points will lie on the hyperplanes. And because $\mu<m$ and the arrangement of the hyperplanes, each $E_i$ is a subset of a scrambled $n$-simplex of size $m$. And finally, one can use a simple recursive argument to show that $\mathbf E$ is (fully) exceptional.
\end{proof}

\begin{example}\label{strictex}
In Figure \ref{strictfig}, we demonstrate Proposition \ref{strictprop} for the triple $(2,D(11),3^{\times16})$ by exhibiting the prescribed exceptional partition.
\begin{figure}
\centering
\includegraphics{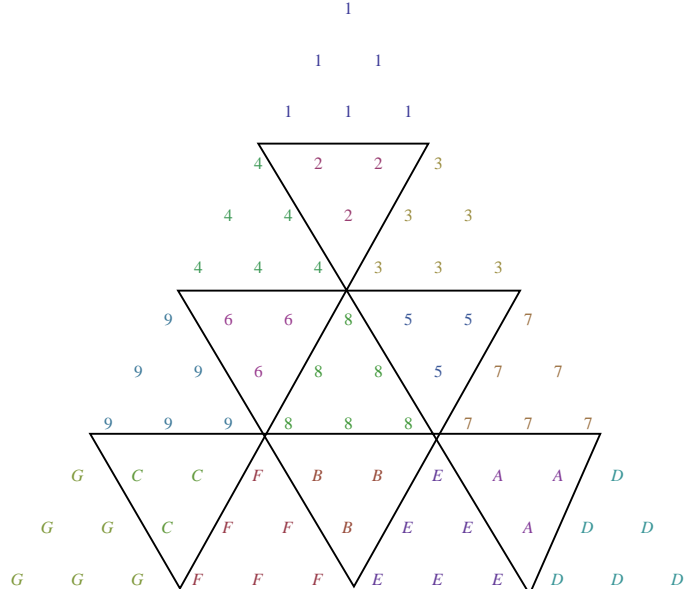}
\caption{An fully exceptional partition for $(2,D(11),3^{\times16})$. $E_i$ is denoted by an $i$ for $i=1,\ldots,9$ and $A,\ldots,G$ for $10,\ldots,16$ respectively.}\label{strictfig}
\end{figure}
\end{example}

Intuitively, the addition of $\delta$ in the proof is designed to eliminate the possibility that points of $D(d)$ lie on the hyperplanes---to avoid ``borderline'' points. When $d=ms$, we have no ``buffer'', so adding $\delta>0$ may give us parts of our partition which are too big and therefore not a subset of a scrambled $n$-simplex of size $m$. So when $d=ms$, we use almost the same construction for our partitions, but we make careful choices about where to send the borderline points (compare Examples \ref{strictex} and \ref{knownnagataex}). In the cases excluded by Theorem \ref{nagatagen}, it is not possible to make these choices and end up with an exceptional partition, but in all other cases it is. We do this explicitly for our base cases, which sets us up for induction on $n$.

\begin{lemma}\label{basecaselem}
The triple $\left(n,D(sm),m^{\times s^n}\right)$ admits an exceptional partition in which each part is a subset of a scrambled $n$-simplex of size $m$ in the cases:
\begin{enumerate}
\item $n=2$, $s\geq4$
\item $n=3$, $s=3$
\item $n=4$, $s=2$.
\end{enumerate}
\end{lemma}

\begin{proof}[Sketch of Proof]
For the case of $n=2$, $s\geq4$, we show that one is able to reduce to Proposition \ref{strictprop}. We construct the partition in two steps. The first $6s-9$ parts form a ``border'' around $D(sm)$, whose complement is a translation of $D(sm-3m-3)$. Then the remaining $(s-3)^2$ parts can then by constructed using Proposition \ref{strictprop} since $sm-3m-3<(s-3)m$. Rather than giving all of the details, we refer the reader to Example \ref{knownnagataex} for an illustration.

For the other two cases, explicit descriptions of $\mathbf E$ are also possible. 
\end{proof}

\begin{example}\label{knownnagataex}
In Figure \ref{knownnagatafig} we demonstrate the first part of Lemma \ref{basecaselem} by illustrating the first step in the prescribed full generalized reduction algorithm for $\left(2, D(15),3^{\times 25}\right)$. We see that the partition is exceptional because no other sextuple of points has the same centroid as $E_1$, no pair of disjoint sextuples of points $F_2$, $F_3$ with $\sigma(F_2),\sigma(F_3)\neq0$, have the same centroids as $E_2$ and $E_3$ respectively, etc. The remaining points form a translation of $D(3)$, and by Proposition \ref{strictprop}, there exists an exceptional partition of $(2,D(3),(3^{\times4}))$. Notice that the key to making this construction work is that $s\geq4$ so that we can reduce to Proposition \ref{strictprop}.
\begin{figure}
\centering
\includegraphics{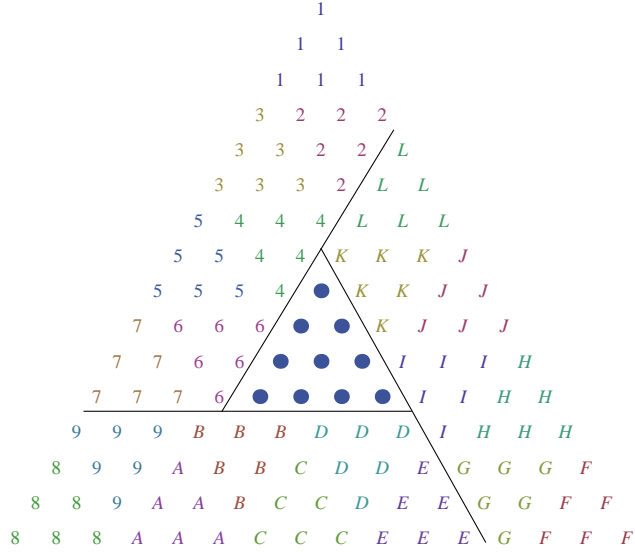}
\caption{An exceptional partition of $\left(2,D(15),3^{\times 21}\right)$.}\label{knownnagatafig}
\end{figure}
\end{example}

Finally we prove the inductive step, which gives rise to the theorem.

\begin{proof}[Sketch of Proof of Theorem \ref{nagatagen}]
The case of $d<sm$ already being covered, we fix $m$ and $s$, let $d=sm$, and induct on $n$. Given the base cases from Lemma \ref{basecaselem}, this will prove the theorem.

We again construct the partition $\mathbf E$ in two steps. The first $s^n-(s-1)^n$ parts will partition $D_1=\{\mathbf a\in D(d)|a_n\leq m\}$. The complement $D(d)\smallsetminus D_1$ is then a translation of $D(sm-m-1)$, of which we can construct the remaining $(s-1)^n$ parts of the partition by Proposition \ref{strictprop}.

So it remains to construct a partition of $D_1$ with each of the $s^n-(s-1)^n$ parts a subset of a scrambled $n$-simplex of size $m$. Let $D'=\{\mathbf a\in D^n(d)|a_n=0\}$. Then considering $D'$ as a subset of $\N^n$, we see that it is exactly $D^{n-1}(sm)$. By the inductive hypothesis, there exists an exceptional partition $(E'_1,\ldots,E_{s^{n-1}}')$ of $D'$ with each $E'_k$ a subset of a scrambled $(n{-}1)$-simplex of size $m$.

Let $\pi$ be the projection of $D^n(sm)$ onto $D'$ via $(a_0,\ldots,a_n)\mapsto(a_0,\ldots,a_{n-1},0)$. Let $\hat E'_k=\pi^{-1}(E'_k)\cap D_1$. Each $\hat E'_k$ then has the form
$$\hat E'_k\subseteq (\text{a scrambled $(n{-}1)$-simplex of size $m$})\times\{0,\ldots,m\}.$$
We can subdivide the set on the right-hand side into subsets of scrambled $n$-simplices of size $m$, which will induce a partition of $\hat E'_k$. Collecting all such parts from all of the $\hat E'_k$, we end up with the desired full exceptional partition of $D_1$.
\end{proof}
\bibliographystyle{amsplain}
\bibliography{NagataBib}
\end{document}